\documentclass{article}%
\usepackage{amsfonts}
\usepackage{amsmath}
\usepackage{amssymb}
\usepackage{comment}
\usepackage{caption}
\RequirePackage{doi}
\usepackage{hyperref}
\hypersetup{linktoc = all}
\usepackage{tikz}
\usetikzlibrary{positioning}
\usetikzlibrary{arrows,decorations.markings}
\usetikzlibrary{shapes.geometric, arrows}
\usepackage{arydshln}
\usepackage{graphicx}
\usepackage{float}
\providecommand{\U}[1]{\protect\rule{.1in}{.1in}}

\newtheorem{theorem}{Theorem}

\newtheorem{definition}[theorem]{Definition}

\newtheorem{lemma}[theorem]{Lemma}

\newenvironment{proof}[1][Proof]{\noindent\textbf{#1.} }{\ \rule{0.5em}{0.5em}}
\makeatother

\usepackage{titlesec}
\titleformat{\paragraph}
{\normalfont\normalsize\bfseries}{\theparagraph}{1em}{}
\titlespacing*{\paragraph}
{0pt}{3.25ex plus 1ex minus .2ex}{1.5ex plus .2ex}
\definecolor{AnthonyComments}{HTML}{00BFFF}

\begin{document}
\sloppy
\captionsetup[figure]{labelfont={bf},name={Fig.},labelsep=none}

\title{A Short Proof of the Symmetric Determinantal Representation of Polynomials}
\author{Anthony Stefan\footnote{Email: astefan2015@my.fit.edu}\; and Aaron Welters\footnote{Email: awelters@fit.edu}\\Florida Institute of Technology\\Department of Mathematical Sciences\\Melbourne, FL, USA}
\date{\today}
\maketitle

\begin{abstract}
We provide a short proof of the theorem that every real multivariate polynomial has a symmetric determinantal representation, which was first proved in J. W. Helton, S. A. McCullough, and V. Vinnikov, \textit{Noncommutative convexity arises from linear matrix inequalities}, J. Funct. Anal. 240 (2006), 105-191. We then provide an example using our approach and extend our results from the real field $\mathbb{R}$ to an arbitrary field $\mathbb{F}$ different from characteristic $2$. The new approach we take is only based on elementary results from the theory of determinants, the theory of Schur complements, and basic properties of polynomials.
\end{abstract}

\section{Introduction}
As was first proven in 2006 by J. W. Helton, S. McCullough, and V. Vinnikov \cite{06HMV}, any real polynomial $p(z)\in\mathbb{R}[z]$ in $n$ variables [$z=(z_1,\ldots, z_n)$] has a \textit{symmetric determinant representation}, i.e., there exists an affine linear matrix pencil $A_0+\sum_{i=1}^nz_iA_i$ with symmetric matrices $A_0,\ldots, A_n\in \mathbb{R}^{m\times m}$ such that
\begin{align}
p(z)=\det \left(A_0+\sum_{i=1}^nz_iA_i\right).
\end{align}
We give a new proof of this theorem, which we will refer to as the \textit{HMV Theorem}. 

One of the merits of our proof that makes it short and elementary is that it requires no prior knowledge of multidimensional systems theory (as compared to \cite{06HMV} and \cite{12RQ}) or advanced representation theory for multivariate polynomials (as compared to \cite{11BG}). For instance, the proof of the HMV theorem in both \cite{06HMV} and \cite{12RQ}, it is necessary to know or first prove that an arbitrary real polynomial has a ``linear description." Thus, in \cite[p. 106]{06HMV} they say that ``Our determinantal representation theorem is a bi-product of the theory of systems realizations of noncommutative rational functions..." (cf. Theorem 14.1 and its proof in \cite{06HMV}), whereas R. Quarez in \cite{12RQ} gives a more elementary proof of the HMV Theorem, yet has to first derive a result on symmetrizable unipotent linear descriptions for homogeneous polynomials (see Proposition 4.2 and the proof of Theorem 4.4 in \cite{12RQ}). 

An alternative proof of the HMV Theorem was given by B. Grenet et al. in \cite{11BG} which is based on symmetrizing the algebraic complexity theoretic construction by L. Valiant in \cite{79LV}. More precisely, it is proved in \cite{79LV} that every polynomial has a nonsymmetric determinantal representation which is proved using a weighted digraph construction and then in \cite{11BG} they use a similar yet symmetrized construction to get a symmetric determinantal representation to prove the HMV Theorem. As such, the proof in \cite{11BG} requires a bit of effort to derive the HMV Theorem as there is quite a bit of prior knowledge needed to do this construction, especially if you are not familiar with algebraic complexity theory.

In this paper we give a short proof of the HMV Theorem (see Theorem \ref{ThmSymmDetRepr} and its proof) using some elementary results from the theory of determinants and the theory of Schur complements (see Lemmas \ref{LemDetDirectSum}, \ref{LemSchursDetFormula}, \ref{LemSumOfSchurComps}, and \ref{LemRealizSimpleProducts}) and basic properties of polynomials (see Lemma \ref{LemmaPolyRealization}). We then provide an example that not only demonstrates our approach, but also illustrates the known problem of the HMV theorem not being true in a field of characteristic $2$ (e.g., $\mathbb{F}_2$). As such, we provide a discussion on this and then extend our results to arbitrary fields different from characteristic $2$.

The rest of the paper will proceed as follows. In Sec. \ref{SecPrelims} we establish the notation, definitions, and preliminary results used in the paper. In Sec. \ref{SecSymmDetRepr} we prove the HMV Theorem. In Sec. \ref{SecExample} we provide an example using our approach. In Sec. \ref{SecExtensions} we discuss the extension of the HMV Theorem to arbitrary fields different from characteristic $2$. Finally, in Sec. \ref{SecAppendix} we provide an Appendix with auxiliary results.

\section{Preliminaries}\label{SecPrelims}
We will denote any matrix $A\in\mathbb{R}^{m\times m}$ that is partitioned in $2\times2$ block matrix form by 
\[A = [A_{ij}]_{i,j=1,2} = \begin{bmatrix}
A_{11} &A_{12} \\ A_{21} &A_{22}
\end{bmatrix}=\left[\begin{array}{c; {2pt/2pt} c }
A_{11} & A_{12} \\\hdashline[2pt/2pt]
A_{21} & A_{22}
\end{array}\right],\] 
where the matrix $A_{ij}\in \mathbb{R}^{m_i\times m_j}$ is called the $(i,j)$-block of $A$. 

The direct sum $A\oplus B$ of two square matrices $A\in\mathbb{R}^{k\times k}$ and $B\in\mathbb{R}^{p\times p}$ is defined to be the matrix $A\oplus B\in\mathbb{R}^{(k+p)\times (k+p)}$ with the $2\times2$ block matrix form
\[
A\oplus B=%
\begin{bmatrix}
A & 0\\
0 & B
\end{bmatrix}
.
\]
The following result is well-known (and is true for matrices with entries arbitrary field $\mathbb{F}$ not just $\mathbb{R}$, see \cite[Sec. 4.3, Exercise 21]{02FIS}).
\begin{lemma}[Determinant of a direct sum]\label{LemDetDirectSum}
If $A$ and $B$ are real symmetric matrices then the direct sum $A\oplus B$ is a real symmetric matrix and
\begin{align}
    \det(A\oplus B)=(\det A)(\det B).
\end{align}
\end{lemma}

The Schur complement of a matrix $A=[A_{ij}]_{i,j=1,2}$ with respect to $A_{22}$ [i.e., with respect to its $(2,2)$-block $A_{22}$] will be denoted as $A/A_{22}$ and defined by 
\[ A /A_{22} = A_{11} - A_{12} A_{22}^{-1} A_{21},\]
whenever the matrix $A_{22}$ is invertible. The following result is also well-known (and is true for matrices with entries in an arbitrary field $\mathbb{F}$ not just $\mathbb{R}$, see \cite[p. 19, Theorem 1.1]{05FZ}).
\begin{lemma}[Schur's determinant formula]\label{LemSchursDetFormula}
If $A=[A_{ij}]_{i,j=1,2}\in\mathbb{R}^{N\times N}$ and $A_{22}$ is invertible then
\begin{align}
    \det A = \det(A_{22})\det (A/A_{22}).
\end{align}
\end{lemma}
The next result is also well-known (see, for instance, \cite[p. 16, Lemma 6]{20SW} although the statement and proof is the valid for any field $\mathbb{F}$ including $\mathbb{R}$).
\begin{lemma}[Sum of a Schur complement with matrix]\label{LemSumOfSchurComps}
If $B\in\mathbb{R}^{m\times m}$, $A=[A_{ij}]_{i,j=1,2}\in\mathbb{R}^{N\times N}$, and $A_{22}$ is invertible with $A/A_{22}\in\mathbb{R}^{m\times m}$ then,
\begin{align}
C/C_{22}=A/A_{22}+B,
\end{align}
where $C=[C_{ij}]_{i,j=1,2}\in\mathbb{R}^{N \times N}$ is partitioned in $2\times2$ block matrix form as
\begin{equation}
C=\begin{bmatrix}
C_{11} & C_{12}\\
C_{21} & C_{22}
\end{bmatrix}=
\begin{bmatrix}
A_{11}+B & A_{12}\\
A_{21} & A_{22}\\
\end{bmatrix}\label{LemSumSchurComplPlusMatrixCMatrixFormula}
\end{equation}
and $C_{22}=A_{22}$ is invertible. Moreover, if both matrices $A$ and $B$ are symmetric then the matrix $C$ is symmetric.
\end{lemma}

The next lemma is proved in the Appendix (in fact, the lemma is actually valid for any field $\mathbb{F}$ different from characteristic $2$, but the proof that we give for $\mathbb{R}$ is much shorter; for more details on the proof of this for such fields $\mathbb{F}$ see Sec. \ref{SecExtensions}).
\begin{lemma}[Realization of simple products]\label{LemRealizSimpleProducts}
The matrix polynomial $uvB$ in any two variables $u$ and $v$ with symmetric matrix $B\in\mathbb{R}^{m\times m}$, has a Bessmertny\u{\i} realization
\begin{equation}
A(u,v)/A_{22}(u,v)=uvB,
\end{equation}
with an affine linear matrix pencil
\begin{equation}
A(u,v)=A_0+uA_1+vA_2=\begin{bmatrix}
A_{11}(u,v) & A_{12}(u,v)\\
A_{21}(u,v) & A_{22}(u,v)
\end{bmatrix},
\end{equation}
such that, for some $N\in\mathbb{N}$, the matrices $A_j\in\mathbb{R}^{N\times N}$ for each $j=0,1,2$ are symmetric and $A_{22}(u,v)$ is a real constant invertible matrix.
\end{lemma}

Although the next lemma (which is the key to the short proof of the HMV Theorem in this paper) and its proof is somewhat obvious and likely known, we provide a proof of it in the Appendix for completeness (in fact, the statement and proof is valid over any field $\mathbb{F}$ not just for $\mathbb{R}$). Before we state it, we need the following definition.
\begin{definition}[Simple product substitution]\label{DefSimpleProdSubst}
For any real polynomial $q=q(z)\in\mathbb{R}[z]$ in $n$ variables $z=(z_1,\ldots, z_n)$, any $k\in \{1,\ldots,n\}$, and any pair of variables $u$ and $v$, the real polynomial $p=q(z)|_{z_k=uv}$ obtained from $q(z)$ by making the substitution $z_k=uv$ is called a \textbf{simple product substitution} on $q$.
\end{definition}

\begin{lemma}[Polynomial realization]\label{LemmaPolyRealization}
Each real polynomial $p$ can be constructed from an affine linear real polynomial $q$ by applying a finite number of simple product substitutions to $q$.
\end{lemma}

\section{Symmetric Determinantal Representation}\label{SecSymmDetRepr}
In this section we will prove the HMV Theorem over the field $\mathbb{R}$ (i.e., Theorem \ref{ThmSymmDetRepr}). To do this we will need the following two lemmas (the first lemma and its proof are valid over any field $\mathbb{F}$ not just $\mathbb{R}$; the second lemma and its proof are valid over any field $\mathbb{F}$ different from characteristic $2$).

\begin{lemma}[Scalar multiplication]\label{LemScalarMultDetSymDetRepr}
If $c$ is a real number and $p(z)\in\mathbb{R}[z]$ has a symmetric determinantal representation then $cp(z)$ also has a symmetric determinantal representation.
\end{lemma}
\begin{proof}
 Suppose $c\in \mathbb{R}$ and $p(z) = \det\left[A(z)\right]$, where $A(z)=A_0+\sum_{i=1}^nz_iA_i$ is an affine linear matrix pencil with symmetric matrices $A_0,\ldots, A_n\in\mathbb{R}^{m\times m}$. Then it follows from Lemma \ref{LemDetDirectSum} that
 \begin{align*}
     B(z)=A(z)\oplus [c]=A_0\oplus [c]+\sum_{i=1}^nz_i(A_i\oplus [c])
 \end{align*}
 is an affine linear matrix pencil with symmetric matrices $A_i\oplus [c]\in \mathbb{R}^{(m+1)\times (m+1)}$ such that
 \begin{align*}
     cp(z) = \det\left[B(z)\right].
 \end{align*}
 This proves that $cp(z)$ has a symmetric determinantal representation.
\end{proof}

\begin{lemma}[Substitution]\label{LemSubSymDetRep}
If $q=q(z)\in\mathbb{R}[z]$ has a symmetric determinant representation and $p$ is obtained from $q$ by a simple product substitution then $p$ also has a symmetric determinant representation. 
\end{lemma}
\begin{proof}
Suppose $q=q(z)$ is a real polynomial in $n$ variables $z=(z_1,\ldots, z_n)$ with a symmetric determinant representation
\begin{align*}
    q(z) = \det\left[A(z)\right],
\end{align*}
where $A(z)=A_0+\sum_{i=1}^nz_iA_i$ is an affine linear matrix pencil with symmetric matrices $A_0,\ldots, A_n\in\mathbb{R}^{m\times m}$. Let $p$ be a polynomial obtained from $q$ by a simple product substitution. Then, by definition, there is a $k\in \{1,\ldots,n\}$ and a pair of variables $u$ and $v$ such that $p$ is the real polynomial $p=q(z)|_{z_k=uv}$ in the variables $z, u,v$. Hence,
\begin{align*}
    p=q(z)|_{z_k=uv}=\det\left[A(z)\right]|_{z_k=uv}=\det\left[A(z)|_{z_k=uv}\right].
\end{align*}
It follows immediately from Lemma \ref{LemSumOfSchurComps} and Lemma \ref{LemRealizSimpleProducts} that
\begin{align*}
    A(z)|_{z_k=uv}=uvA_k+\left(A_0+\sum_{i=1, i\not = k}^nz_iA_i\right)=B(z,u,v)/B_{22}(z,u,v),
\end{align*}
where $B(z,u,v)$ is an affine linear matrix pencil of the form
\begin{equation}
B(z,u,v)=\left(B_0+\sum_{i=1,i\not=k}^nz_iB_i\right)+uB_{n}+vB_{n+1}=\begin{bmatrix}
B_{11}(z,u,v) & B_{12}(z,u,v)\\
B_{21}(z,u,v) & B_{22}(z,u,v)
\end{bmatrix},
\end{equation}
such that, for some $N\in\mathbb{N}$, the matrices $B_j\in\mathbb{R}^{N\times N}$ for each $j=0,\ldots, n+1$ are symmetric and $B_{22}(z,u,v)$ is a constant invertible matrix. In particular,
\begin{equation}
B_{22}(z,u,v)\equiv \det B_{22}(0,0,0)=d \in \mathbb{R}\setminus\{0\}.
\end{equation}
Therefore, it follows from this, Lemma \ref{LemDetDirectSum}, and Lemma \ref{LemSchursDetFormula} that
\begin{align*}
    p=\det\left[B(z,u,v)/B_{22}(z,u,v)\right]=\frac{1}{d}\det\left[B(z,u,v)\right]=\det\left\{B(z,u,v)\oplus [d^{-1}]\right\},
\end{align*} 
which is a symmetric determinantal representation for $p$. This completes the proof.
\end{proof}

We are now ready to prove the HMV Theorem as an immediate corollary of Lemma \ref{LemmaPolyRealization} and Lemma \ref{LemSubSymDetRep} (the statement and its proof is valid for any field $\mathbb{F}$ different from characteristic $2$; for more details on this see Sec. \ref{SecExtensions}).
\begin{theorem}[Symmetric determinantal representation]\label{ThmSymmDetRepr}
Any real polynomial $p(z)\in\mathbb{R}[z]$ in $n$ variables $z=(z_1,\ldots, z_n)$ has a symmetric determinant representation, i.e., there exists an affine linear matrix pencil $A_0+\sum_{i=1}^nz_iA_i$ with real symmetric matrices $A_0,\ldots, A_n\in \mathbb{R}^{m\times m}$ such that
\begin{align}
p(z)=\det \left(A_0+\sum_{i=1}^nz_iA_i\right).
\end{align}
\end{theorem}
\begin{proof}
Let $p=p(z) \in \mathbb{R}[z]$ be a real polynomial in $n$ variables  $z=(z_1,\ldots,z_n)$. If $p$ is an affine linear real polynomial then $p(z)=\det [p(z)]$ is symmetric determinantal representation. Suppose that $p$ is not an affine linear real polynomial. Then by Lemma \ref{LemmaPolyRealization}, there exists an affine linear real polynomial $q$ such that $p=S_l\cdots S_1q$, where $S_k$ is an operation of simple product substitution for each $k=1,\ldots, l$ for some $l\in\mathbb{N}$. By Lemma \ref{LemSubSymDetRep}, $S_1q$ has a symmetric determinantal representation. If $l=1$ then we are done. Thus, assume $l\geq 2$. If $S_k\cdots S_1q$ has a symmetric determinantal representation for some integer $1\leq k< l$ then by Lemma \ref{LemSubSymDetRep}, $S_{k+1}S_k\cdots S_1q=S_{k+1}(S_k\cdots S_1q)$ has a symmetric determinantal representation. Therefore, by induction $p=S_l\cdots S_1q$ has a symmetric determinantal representation. This proves the theorem.
\end{proof}

\section{Example}\label{SecExample}

Here we will use the results in this paper to show how to produce a symmetric determinantal representation for the real polynomial
\begin{align}\label{examplepoly}
    p(z_1,z_2,z_3)=z_1+z_2z_3.
\end{align}
First, (\ref{examplepoly}) can be obtained from the affine linear real polynomial
\begin{align*}
    q(z_1,w_1)=z_1+w_1
\end{align*}
by making the simple product substitution $w_1=z_2z_3$ since
\begin{align*}
    p(z_1,z_2,z_3)=z_1+z_2z_3=q(z_1,w_1)|_{w_1=z_2z_3}.
\end{align*}
Next, $q(z_1,w_1)$ has the symmetric determinantal representation
\begin{align*}
    q(z_1,w_1)=\det\begin{bmatrix}q(z_1,w_1)\end{bmatrix}=\det(z_1[1]+w_1[1])
\end{align*}
and hence,
\begin{gather*}
     p(z_1,z_2,z_3)=z_1+z_2z_3=\det\begin{bmatrix}q_0(z_1,w_1)\end{bmatrix}|_{w_1=z_2z_3}=\det(z_1[1] + z_2z_3[1]).
\end{gather*}
Next, by Lemma \ref{LemSchursDetFormula} and Lemma \ref{LemSumOfSchurComps}, and following the proofs of Lemma \ref{LemRealizSimpleProducts} and Lemma \ref{LemScalarMultDetSymDetRepr} we have
\begin{gather*}
    \det(z_1[1] + z_2z_3[1])\\
    =\det\left\{ [z_1] + \left.\left[\begin{array}{c; {2pt/2pt} c c}
0 & \frac{1}{2}(z_2+z_3) & \frac{1}{2}(z_2 - z_3) \vspace{0.1cm} \\ \hdashline[2pt/2pt] 
\frac{1}{2}(z_2+z_3) & -1 & 0 \\ 
\frac{1}{2}(z_2 - z_3) & 0 & 1
\end{array}\right]\right/ \begin{bmatrix} -1 & 0\\ 0 & 1 \end{bmatrix}\right\}\\
    =\det\left\{\left.\left[\begin{array}{c; {2pt/2pt} c c}
z_1 & \frac{1}{2}(z_2+z_3) & \frac{1}{2}(z_2 - z_3) \vspace{0.1cm} \\ \hdashline[2pt/2pt] 
\frac{1}{2}(z_2+z_3) & -1 & 0 \\ 
\frac{1}{2}(z_2 - z_3) & 0 & 1
\end{array}\right]\right/ \begin{bmatrix} -1 & 0\\ 0 & 1 \end{bmatrix}\right\}\\
=\left(\det \begin{bmatrix} -1 & 0\\ 0 & 1 \end{bmatrix}\right)^{-1}\det\begin{bmatrix}
z_1 & \frac{1}{2}(z_2+z_3) & \frac{1}{2}(z_2 - z_3)  \\
\frac{1}{2}(z_2+z_3) & -1 & 0 \\ 
\frac{1}{2}(z_2 - z_3) & 0 & 1
\end{bmatrix}\\
=(-1)\det\begin{bmatrix}
z_1 & \frac{1}{2}(z_2+z_3) & \frac{1}{2}(z_2 - z_3)  \\
\frac{1}{2}(z_2+z_3) & -1 & 0 \\ 
\frac{1}{2}(z_2 - z_3) & 0 & 1
\end{bmatrix}\\
=\det\left\{\begin{bmatrix}
z_1 & \frac{1}{2}(z_2+z_3) & \frac{1}{2}(z_2 - z_3)  \\
\frac{1}{2}(z_2+z_3) & -1 & 0 \\ 
\frac{1}{2}(z_2 - z_3) & 0 & 1
\end{bmatrix}\oplus [-1]\right\}\\
=\det\begin{bmatrix}
z_1 & \frac{1}{2}(z_2+z_3) & \frac{1}{2}(z_2 - z_3) & 0 \\
\frac{1}{2}(z_2+z_3) & -1 & 0 & 0\\ 
\frac{1}{2}(z_2 - z_3) & 0 & 1 & 0\\
0 & 0 & 0 & -1
\end{bmatrix}.
\end{gather*}
Therefore, the real polynomial $p(z_1,z_2,z_3)=z_1+z_2z_3$ has the symmetric determinantal representation
\begin{align}
   p(z_1,z_2,z_3)=z_1+z_2z_3=\det \left(A_0+z_1A_1+z_2A_2+z_3A_3\right), 
\end{align}
with the affine linear matrix pencil
\begin{align}
   A_0+z_1A_1+z_2A_2+z_3A_3= \begin{bmatrix}
z_1 & \frac{1}{2}(z_2+z_3) & \frac{1}{2}(z_2 - z_3) & 0 \\
\frac{1}{2}(z_2+z_3) & -1 & 0 & 0\\ 
\frac{1}{2}(z_2 - z_3) & 0 & 1 & 0\\
0 & 0 & 0 & -1
\end{bmatrix}
\end{align}
and symmetric matrices $A_0,A_1,A_2,A_3\in\mathbb{R}^{4\times 4}$ given by
\begin{gather}
    A_0=\begin{bmatrix}
    0 & 0 & 0 & 0\\
    0 & -1 & 0 & 0\\
    0 & 0 & 1 & 0\\
    0 & 0 & 0 & -1
    \end{bmatrix},\;
    A_1=\begin{bmatrix}
    1 & 0 & 0 & 0\\
    0 & 0 & 0 & 0\\
    0 & 0 & 0 & 0\\
    0 & 0 & 0 & 0
    \end{bmatrix},\\
    A_2=\begin{bmatrix}
    0 & \frac{1}{2} & \frac{1}{2} & 0\\
    \frac{1}{2} & 0 & 0 & 0\\
    \frac{1}{2} & 0 & 0 & 0\\
    0 & 0 & 0 & 0
    \end{bmatrix},\;
    A_3=\begin{bmatrix}
    0 & \frac{1}{2} & -\frac{1}{2} & 0\\
    \frac{1}{2} & 0 & 0 & 0\\
    -\frac{1}{2} & 0 & 0 & 0\\
    0 & 0 & 0 & 0
    \end{bmatrix}.
\end{gather}

\section{Extensions}\label{SecExtensions}
In this section we will discuss how to extended our proof of the HMV Theorem (i.e., Theorem \ref{ThmSymmDetRepr}) on symmetric determinantal representations of real polynomials to polynomials with coefficients in an arbitrary field $\mathbb{F}$ different from characteristic $2$. 

We will begin in Subsection \ref{SecPreviousResults} with a discussion of some previous results in this regard and compare it to ours. Furthermore, we include a brief discussion on how we modify our proof of the HMV Theorem from the field $\mathbb{R}$ to an arbitrary field $\mathbb{F}$ different from characteristic $2$. Moreover, we explain why the proof of our theorem would fail for fields of characteristic $2$ by pointing out which step in our proof causes the problem and illustrate the issue in the example from Section \ref{SecExample}. Finally, in Subsection \ref{SecExtensionOfHMVThm} we give our short proof of the extension of the HMV Theorem to such fields.

\subsection{Comparison to previous results}\label{SecPreviousResults}
In \cite[Sec. 5.1]{12RQ}, R. Quarez extends the HMV Theorem to polynomials with coefficients in an arbitrary ring $R$ of characteristic different from $2$ (see \cite[Theorem 5.1]{12RQ}). To do this, he has to adjust the steps in his construction to work not only over $\mathbb{R}$, but also over the ring $R$. The main issue he has to resolve is adjusting his proof (mainly of \cite[Theorem 4.4]{12RQ}) to avoid issues with inversion and using the diagonalization theorem for real symmetric matrices. However, these matrices could involve, for instance, square roots and hence not allowable over a general ring (cf. \cite[Theorem 5.1]{12RQ} and its proof). 

In comparison with our paper, although we just treat a field $\mathbb{F}$ different from characteristic $2$ instead of a ring $R$ (for the sake of simplicity), we only have to adjust one step in our proof of the HMV Theorem. The adjustment we have to make is in the proof of Lemma \ref{LemRealizSimpleProducts}, namely, for the case when the symmetric matrix $B\in\mathbb{F}^{m\times m}$ is not invertible. This is due to the issue that it is not always possible to find a scalar $\lambda_0\in\mathbb{F}\setminus\{0\}$ such that $B-\lambda_0 I_m$ is invertible. This issue would not occur if $\mathbb{F}$ is an infinite field (i.e., a field of characteristic $0$) as $B$ then can only have a finite number of eigenvalues \cite[p. 249, Theorem 5.3]{02FIS}, but it is an issue when $\mathbb{F}$ is a finite field (i.e., a field of characteristic $p$, for some prime number $p$). To see this, let $\mathbb{F}$ be a finite field. Then, it has $p^k$ elements in it for some prime number $p$ (so that $\mathbb{F}$ has characteristic $p$) and some integer $k\geq 1$. Let $0,\lambda_1,\ldots, \lambda_{p^k-1}$ be all the distinct elements of $\mathbb{F}$ and consider the diagonal matrix $B=\operatorname{diag}\{0,\lambda_1,\ldots, \lambda_{p^k-1}\}\in \mathbb{F}^{p^k\times p^k}$. This is a symmetric matrix in $\mathbb{F}^{p^k\times p^k}$ that is not invertible such that $B-\lambda_0 I_{p^k}$ is not invertible for every scalar $\lambda_0\in\mathbb{F}\setminus\{0\}$.

In \cite{11BG}, B. Grenet et al. gives a proof of the HMV Theorem using a construction that, from the very beginning, was meant to be valid for any field of characteristic different from $2$. As they point out (see \cite[Sec. 5]{11BG}), their constructions are not valid for fields of characteristic $2$ because they need to use the scalar $1/2$ (i.e., the multiplicative inverse of $2$ which, of course, does not exist in a field of characteristic $2$).

Comparing our paper, the reason that our construction will fail as well, for a field $\mathbb{F}$ of characteristic $2$, is for a similar reason as in \cite{11BG}. More precisely, the only issue in our proof of the HMV Theorem is in the proof of Lemma \ref{LemRealizSimpleProducts}, where in order to realize the simple product $uvB$ for a symmetric matrix $B\in\mathbb{F}^{m\times m}$, we need to use the scalar $1/2$.

In \cite{13GMT}, B. Grenet, T. Monteil, and S. Thomass\'{e} study the problem of the symmetric determinantal representations of polynomials with coefficients in fields with characteristic $2$, i.e., on the extension of the HMV Theorem to such fields. One of the motivations for their paper was to prove a conjecture from \cite{11BG} that such representations do not always exist for such fields. This conjecture is proven in \cite[Sec. 4]{13GMT} and, in particular, \cite[Sec. 4.2]{13GMT} gives as an example the polynomial $xy+z$ in the three variables $x,y,z$ from the field $\mathbb{F}_2$, the field with two elements, and prove it has no symmetric determinantal representation. More specifically, they prove by their result \cite[Theorem 4.2]{13GMT}, that the polynomial $p(x,y,z)=xy+z$ can \textit{not} be represented as the determinant of a symmetric matrix with entries in $\mathbb{F}_2\cup \{x,y,z\}$.

Therefore, if you compare this to our example from Section \ref{SecExample}, since we need to use the scalar $1/2$ as well, we face the same issue if our field was indeed characteristic $2$.

\subsection{Extension of the HMV Theorem}\label{SecExtensionOfHMVThm}
In this subsection we will prove an extension of the HMV Theorem from the field $\mathbb{R}$ (i.e., Theorem \ref{ThmSymmDetRepr}) to an arbitrary field $\mathbb{F}$ different from characteristic $2$ (i.e., Theorem \ref{ThmSymmDetReprExt}). We will need the following three elementary lemmas (which are true over any field $\mathbb{F}$ with any characteristic).

The following lemma is well-known (see, for instance, \cite[p. 123, Theorem 3]{72PP} as proved in \cite[p. 57, Theorem 1.1$^\prime$]{66KP}, cf. \cite[pp. 48-49, Theorem 1.1]{66KP} and its proof).
\begin{lemma}[Rank factorization]\label{LemRankCongruenceFactorization}
If $\mathbb{F}$ is a field and $A\in \mathbb{F}^{m\times m}$ is a symmetric matrix with rank $r\geq 1$ then there exists an invertible matrix $Y\in \mathbb{F}^{m\times m}$ and a invertible symmetric matrix $B\in \mathbb{F}^{r\times r}$ such that
\begin{align}\label{RankFactorizationFormula}
    A=Y^T\begin{bmatrix}
    B & 0\\
    0 & 0_{m-r}
    \end{bmatrix}Y,
\end{align}
where the zero matrices bordering $B$ are absent if $r=m$.
\end{lemma}

The next two lemmas are well-known (see, for instance, \cite[p. 17, Lemma 8]{20SW} and \cite[p. 21, Proposition 12]{20SW} in the field $\mathbb{C}$ although the proof is the same for any field $\mathbb{F}$).
\begin{lemma}[Shorted matrices are Schur complements]\label{LemShortedMatrSchuCompl}
If $\mathbb{F}$ is a field and $B=[B_{ij}]_{i,j=1,2}\in
\mathbb{F}
^{(r+k)\times (r+k)}$ with $B_{22}\in
\mathbb{F}
^{k\times k}$ invertible and $B/B_{22}\in
\mathbb{F}
^{r\times r}$ then the direct sum $B/B_{22} \oplus 0_l\in\mathbb{F}^{\left(r+l\right)  \times\left(r+l\right)}$ is a Schur complement
\begin{align}
C/C_{22}&=B/B_{22} \oplus 0_l=
\begin{bmatrix}
B/B_{22} & 0\\
0 & 0_l
\end{bmatrix},
\end{align}
where $C\in
\mathbb{F}
^{\left(  r+l+k\right)  \times\left(  r+l+k\right)  }$ is a $3\times3$ block matrix with the following block partitioned structure $C=[C_{ij}]_{i,j=1,2}$: 
\begin{align} 
C & =
\left[\begin{array}{c;{2pt/2pt} c c}
C_{11} & C_{12} \\ \hdashline[2pt/2pt]
C_{21} & C_{22}
\end{array}\right]
=
\left[\begin{array}{c c; {2pt/2pt} c}
B_{11} & 0 & B_{12} \\
0 & 0_l & 0 \\ \hdashline[2pt/2pt]
B_{21} & 0 & B_{22}
\end{array}\right],\label{LemShortedMatricesAreSchurComplMatrixC}
\end{align}
and $C_{22}=B_{22}$ is invertible. Moreover, if the matrix $B$ is symmetric then the matrix $C$ is symmetric.
\end{lemma}

\begin{lemma}[Matrix multiplication of a Schur complement]\label{LemMatrixMultOfSchurCompl}
If $\mathbb{F}$ is a field and $C=[C_{ij}]_{i,j=1,2}\in
\mathbb{F}
^{(m+k)\times (m+k)}$ with $C_{22}\in
\mathbb{F}
^{k\times k}$ invertible and $C/C_{22}\in
\mathbb{F}
^{m\times m}$ then, for any matrices $X\in \mathbb{C}^{s\times m}$ and $Y\in \mathbb{C}^{m\times s}$,
\begin{align}
D/D_{22}=X(C/C_{22})Y,
\end{align}
where $D\in \mathbb{C}^{(s+k)\times (s+k)}$ is the $2\times 2$ block matrix
\begin{align}
D=\begin{bmatrix}
D_{11} & D_{12}\\
D_{21} & D_{22}
\end{bmatrix}= \begin{bmatrix}
XC_{11}Y & XC_{12}\\
C_{21}Y & C_{22}
\end{bmatrix}=\begin{bmatrix}
X & 0\\
0 & I_{k}
\end{bmatrix}\begin{bmatrix}
C_{11} & C_{12}\\
C_{21} & C_{22}
\end{bmatrix}\begin{bmatrix}
Y & 0\\
0 & I_{k}
\end{bmatrix}\label{PropMatrixMultOfSchurComplCMatrix}
\end{align}
and $D_{22}=C_{22}$ is invertible. Moreover, if $C$ is symmetric and $X=Y^T$ then $D$ is symmetric.
\end{lemma}

\begin{theorem}[Symmetric determinantal representation]\label{ThmSymmDetReprExt}
Let $\mathbb{F}$ be a field of characteristic different from $2$. Then any polynomial $p(z)\in\mathbb{F}[z]$ in $n$ variables $z=(z_1,\ldots, z_n)$ has a symmetric determinant representation, i.e., there exists an affine linear matrix pencil $A_0+\sum_{i=1}^nz_iA_i$ with symmetric matrices $A_0,\ldots, A_n\in \mathbb{F}^{m\times m}$ such that
\begin{align}
p(z)=\det \left(A_0+\sum_{i=1}^nz_iA_i\right).
\end{align}
\end{theorem}
\begin{proof}
The proof of this theorem is the same as the proof of the case when the field is $\mathbb{R}$ (i.e., the same proof as we gave for Theorem \ref{ThmSymmDetRepr}, but just replace the field $\mathbb{R}$ in that proof with the field $\mathbb{F}$ of characteristic different from $2$) except we have to make one change to the proof of Lemma \ref{LemRealizSimpleProducts} in the Appendix, namely, for the case when the symmetric matrix $B\in\mathbb{F}^{m\times m}$ is not invertible we need to prove Lemma \ref{LemRealizSimpleProducts} for the matrix polynomial $uvB$. If $B=0$ then the proof is immediate. Thus, assume $B\not=0$. In this case, we prove Lemma \ref{LemRealizSimpleProducts} by first applying Lemma \ref{LemRankCongruenceFactorization} to the matrix $B$ to get the factorization
\begin{align}
    uvB=Y^T\begin{bmatrix}
    uvB_1 & 0\\
    0 & 0_{m-r}
    \end{bmatrix}Y,
\end{align}
where $Y\in\mathbb{F}^{m\times m}$ and $B_1\in\mathbb{F}^{r\times r}$ are invertible matrices and $B_1$ is symmetric. Now we can appeal to the first part of the proof of Lemma \ref{LemRealizSimpleProducts} applied to the matrix polynomial $uvB_1$ and then use Lemma \ref{LemShortedMatrSchuCompl} followed by Lemma \ref{LemMatrixMultOfSchurCompl} to prove Lemma \ref{LemRealizSimpleProducts} for the matrix polynomial $uvB$. This proves the theorem.
\end{proof}

\section{Appendix}\label{SecAppendix}
The following lemma is well-known (see, for instance, \cite[p. 16, Proposition 7]{20SW} although the statement and proof is the valid for any field $\mathbb{F}$ including $\mathbb{R}$).
\begin{lemma}[Sum of two Schur complements]\label{LemSumSchurComp} If $A\in
\mathbb{R}
^{m\times m}$ and $B\in
\mathbb{R}
^{n\times n}$ are $2\times2$ block matrices
\[A=\begin{bmatrix}
A_{11} & A_{12}\\
A_{21} & A_{22}
\end{bmatrix}
,\text{ }B=
\begin{bmatrix}
B_{11} & B_{12}\\
B_{21} & B_{22}
\end{bmatrix}
\]
such that $A_{22}\in
\mathbb{R}
^{p\times p}$, $B_{22}\in
\mathbb{R}
^{q\times q}$ are invertible and $A/A_{22}$, $B/B_{22}\in
\mathbb{R}
^{k\times k}$ then
\begin{align}
C/C_{22}=A/A_{22}+B/B_{22},
\end{align}
where $C\in
\mathbb{R}
^{\left(  k+p+q\right)  \times\left(  k+p+q\right)  }$ is the $3\times3$ block
matrix with the following block partitioned structure $C=[C_{ij}]_{i,j=1,2}$:
\begin{equation}\label{SumOfSchurCompsCMatrix}
C=
\left[\begin{array}{c;{2pt/2pt} c c}
C_{11} & C_{12} \\ \hdashline[2pt/2pt]
C_{21} & C_{22}
\end{array}\right]
=
\left[\begin{array}{c;{2pt/2pt} c c}
A_{11}+B_{11} & A_{12} & B_{12} \\ \hdashline[2pt/2pt]
A_{21} & A_{22} & 0\\
B_{21} & 0 & B_{22}
\end{array}\right],
\end{equation}
and
\begin{equation}\label{SumOfSchurCompsC22Matrix}
C_{22}=
\begin{bmatrix}
A_{22} & 0\\
0 & B_{22}
\end{bmatrix}
\end{equation}
is invertible. Moreover, if both matrices $A$ and $B$ are symmetric then the matrix $C$ is also symmetric.
\end{lemma}

\begin{proof}[Proof of Lemma \ref{LemRealizSimpleProducts}]\label{PrfLemRealizSimpleProducts}
Let $B\in\mathbb{R}^{m\times m}$ be a symmetric matrix and consider the matrix polynomial $uvB$ in the two variables $u$ and $v$. Consider first the case that $B$ is invertible. If $u=v$ then we have
\begin{align}
    u^2B =\left.\left[\begin{array}{c;{2pt/2pt} c}
0 & uI_m \\ \hdashline uI_m & -B^{-1}
\end{array}\right]\right/ \left[-B^{-1}\right],
\end{align} 
where $I_m$ is the $m\times m$ identity matrix in $\mathbb{R}^{m\times m}$. It follows from this that if $u$ and $v$ are independent variables then 
\begin{align}
    \left(\frac{1}{2}(u \pm v)\right)^2= \left.\left[\begin{array}{c;{2pt/2pt} c}
0 & \frac{1}{2}(u\pm v)I_m \\ \hdashline \frac{1}{2}(u\pm v)I_m & -B^{-1}
\end{array}\right]\right/ \left[-B^{-1}\right]
\end{align}
and hence by Lemma \ref{LemSumSchurComp},
\begin{gather}
    uvB=\left(\frac{1}{2}(u + v)\right)^2B+\left(\frac{1}{2}(u - v)\right)^2(-B)\\
    =\left.\left[\begin{array}{c;{2pt/2pt} c c}
0 & \frac{1}{2}(u+ v)I & \frac{1}{2}(u- v)I\\ \hdashline
\frac{1}{2}(u+ v)I & -B^{-1} & 0  \\
\frac{1}{2}(u- v)I & 0 &  B^{-1}
\end{array}\right]\right/ \begin{bmatrix}
-B^{-1} & 0 \\ 0 &  B^{-1}\end{bmatrix}.
\end{gather}
This proves the lemma in the case that $B$ is invertible. Now suppose $B$ is not invertible. Then there exists a $\lambda_{0}\in\mathbb{R}\setminus\{0\}$ which is not an eigenvalue of $B$. Let $B_1=B-\lambda_{0}I_{m}$ and $B_2=\lambda_{0}I_{m}$. Therefore, since both $B_1$ and $B_2$ are invertible symmetric matrices in $\mathbb{R}^{m\times m}$ and $uvB =uvB_1+uvB_2$, the proof then follows immediately from Lemma \ref{LemSumSchurComp}.
\end{proof}

\begin{proof}[Proof of Lemma \ref{LemmaPolyRealization}]
First, Lemma \ref{LemmaPolyRealization} is obviously true for any affine linear real polynomial, i.e., for any real polynomial of degree less than or equal to one. We will now prove it is true for any real polynomial of degree $2$. Let $n\in\mathbb{N}$, $a_0,\ldots, a_{n}, b_1,\ldots, b_{\frac{n(n+1)}{2}}$ be real scalars, and $z_1,\ldots, z_n, w_1,\ldots, w_{\frac{n(n+1)}{2}}$ be independent variables. Consider the affine linear real polynomial
\begin{align*}
    q(z,w)=a_0+\sum_{l=1}^{n} a_lz_l+\sum_{k=1}^{\frac{n(n+1)}{2}} b_kw_k.
\end{align*}
Let $S_{k}=|_{w_k=z_iz_j}$ denote the operation of simple product substitution of $w_k=z_iz_j$, for all integer pairs $i,j$ with $1\leq i\leq j\leq n$ with integers $k=1,\ldots, \frac{n(n+1)}{2}$ ordered by the lexicographical order on the pairs $(i,j)$. Applying the operations consecutively of $S_1,\ldots, S_{\frac{n(n+1)}{2}}$ to $q$ we get the polynomial  
\begin{align*}
    S_{\frac{n(n+1)}{2}}\cdots S_1q=a_0+\sum_{l=1}^{n} a_lz_l+\sum_{k=1}^{\frac{n(n+1)}{2}} b_kz_iz_j.
\end{align*}
This proves the lemma for any real polynomial of degree $2$. Suppose the lemma is true for all real polynomials of degree less than or equal to $d$ for some natural number $d\geq 2$.  We will now prove the lemma is true for any real polynomial of degree $d+1$. Let $n\in\mathbb{N}$, $b_1,\ldots, b_{M_{n,d+1}}$ be real scalars, and $z_1,\ldots, z_n, w_1,\ldots, w_{M_{n,d+1}}$ be independent variables, where $M_{n,d+1}$ is the number of monomials of degree $d+1$ in $n$ independent variables. To be explicit, it is a well-known result that this number is given by the following formula
\begin{align*}
    M_{n,d+1}=\binom{n+d}{d+1}=\frac{(n+d)!}{(d+1)!(n-1)!}.
\end{align*}
Let $q_d(z)$ be a real polynomial of degree less than or equal to $d$ in the $n$ independent variables $z=(z_1,\ldots, z_n)$. Consider the real polynomial $q(z,w)$ of degree less than or equal to $d$ in the $n+M_{n,d+1}$ independent variables $z_1,\ldots, z_n,w_1,\ldots, w_{M_{n,d+1}}$ defined by
\begin{align*}
    r(z,w)=q_d(z)+\sum_{k=1}^{M_{n,d+1}}b_kz^{\alpha_k}w_k,
\end{align*}
where the set $\{z^{\alpha_k}:k=1,\ldots, M_{n,d+1}\}$ equals the set of all monomials in the $n$ variables $z_1,\ldots, z_n$ of degree $d-1$ such that $z^{\alpha_k}z_iz_j$, $k=1,\ldots, M_{n,d+1}, 1\leq i\leq j\leq n$ is a list, with no repeats, of all the monomials in the $n$ variables $z_1,\ldots, z_n$ of degree $d+1$. By the induction hypothesis, there exists a affine linear real polynomial $q$ and a finite number $T_l$, $l=1,\ldots M$ of simple product substitution operations such that
\begin{align*}
    r(z,w)=T_M\cdots T_1 q.
\end{align*}
Let $S_{k}=|_{w_k=z_iz_j}$ denote the operation of simple product substitution of $w_k=z_iz_j$, for all integer pairs $i,j$ with $1\leq i\leq j\leq n$ with integers $k=1,\ldots, \frac{n(n+1)}{2}$ ordered by the lexicographical order on the pairs $(i,j)$. Applying the operations consecutively of $S_1,\ldots, S_{\frac{n(n+1)}{2}}$ to $r(z,w)$ we get the polynomial  
\begin{align*}
    S_{\frac{n(n+1)}{2}}\cdots S_1 r
    = S_{\frac{n(n+1)}{2}}\cdots S_1T_M\cdots T_1q=q_d(z)+\sum_{k=1}^{M_{n,d+1}}b_kz^{\alpha_k}z_iz_j=p(z).
\end{align*}
This proves the lemma for any real polynomial of degree $d+1$. Therefore, by induction the lemma is true for any real polynomial of any degree. This proves the lemma.
\end{proof}

\end{document}